\documentclass[a4paper,12pt,leqno]{article}
\makeatletter
\usepackage{amsfonts,delarray,amssymb,amsmath,amsthm,a4,a4wide}
\usepackage{amsmath, euscript}
\usepackage[ansinew]{inputenc}

\newtheorem{theo}{Theorem}
\newtheorem{pro}{Proposition}[section]
\newtheorem{lem}[pro]{Lemma}

\newtheorem{coro}[pro]{Corollary}
\newtheorem{remark}[pro]{Remark}
\newtheorem{definition}[pro]{Definition}

\numberwithin{equation}{section}

\def\XXint#1#2#3{{\setbox0=\hbox{$#1{#2#3}{\int}$}
     \vcenter{\hbox{$#2#3$}}\kern-.5\wd0}}

\DeclareMathOperator{\supp}{supp}

\DeclareMathOperator{\curl}{curl}

\DeclareMathOperator{\dist}{dist}
\DeclareMathOperator{\mcr}{mcr}

\providecommand{\abs}[1]{\left\vert#1\right\vert}

\providecommand{\pnorm}[2]{\left\Vert#1\right\Vert_{L^{#2}}}
\providecommand{\pnormspace}[3]{\left\Vert#1\right\Vert_{L^{#2}(#3)}}
\providecommand{\pqnorm}[3]{\left\Vert#1\right\Vert_{L^{#2,#3}}}
\providecommand{\pqnormspace}[4]{\left\Vert#1\right\Vert_{L^{#2,#3}(#4)}}

\providecommand{\Rn}[1]{\mathbb{R}^{#1}}

\providecommand{\wnorm}[1]{\left\lvert \mspace{-1.8mu} \left\lvert
\mspace{-1.8mu} \left\lvert #1 \right\rvert \mspace{-1.8mu} \right\rvert
\mspace{-1.8mu} \right\rvert_{L^{2,\infty}}}

\def\dt{\partial_t}

\def\hal{\frac{1}{2}}

\def\vchi{\text{\large{$\chi$}}}

\def\({\left(}
\def\){\right)}
\def\1{\mathbf{1}}

\def\a{\alpha}

\def\B{{\mathcal{B}}}

\def\curl{{\rm curl\,}}

\def\dist{\text{dist}\ }
\def\div{\mathrm{div} \ }

\def\dr{\delta_{\mr}}
\def\dt0{{{\frac{d}{dt}}_{|t=0}}}

\def\F{{\mathcal{F}}}

\def\hal{\frac{1}{2}}

\def\loc{{\text{\rm loc}}}

\def\lti{L^{2,\infty}}

\def\l|{\left|}

\def\mr{\mathbb{R}}

\def\mz{\mathbb{Z}}
\def\nab{\nabla}

\def\np{\nab^{\perp}}

\def\Q{\mathcal{Q}}

\def\r|{\right|}

\def\sm{\setminus}

\def\UR{{\mathbf U_R}}

\def\w{{w_n}}
\def\Q{{\mathbf{P}_n^\beta}}
\def\Z{Z_n^\beta}

\title{Lorentz space estimates for the Coulombian renormalized energy}
\author{Sylvia Serfaty\footnote{Supported by an EURYI Award}\, and Ian Tice\footnote{Supported by an NSF Postdoctoral Research Fellowship}}

\begin{document}
\maketitle

\begin{abstract}
In this paper we obtain optimal estimates for the ``currents'' associated to point masses in the plane, in terms of the Coulombian renormalized energy of Sandier-Serfaty \cite{ss1,ss3}.  To derive the estimates, we use a technique that we introduced in \cite{st}, which couples the ``ball construction method'' to estimates in the Lorentz space $L^{2,\infty}$.
\end{abstract}

%%%%%%%%%%%%%%%%%%%%%%%%%%%%%%%%%%%%%
\section{Introduction}
%%%%%%%%%%%%%%%%%%%%%%%%%%%%%%%%%%%%%

 In \cite{ss1}, Sandier and the first author introduced a
Coulombian ``renormalized energy'' associated to a discrete set of points in
the plane via a vector field $j$.   The simplest setting is that of a vector field  $j: \Rn{2} \to \Rn{2}$ satisfying
\begin{equation}\label{eqj}
\curl j = 2\pi \nu - 1,  \qquad \div j=0
\end{equation}
in the sense of distributions, where $\nu$ has the form
\begin{equation*}
\nu= \sum_{p \in\Lambda} \delta_{p}\quad \text{ for
some discrete set }   \Lambda\subset\mr^2.
\end{equation*}
Then  for any non-negative and compactly supported function $\chi$ we define
\begin{equation}\label{WR}W(j, \chi) = \lim_{\eta\to 0} \(
\hal\int_{\mr^2 \backslash \cup_{p\in\Lambda} \bar{B}(p,\eta) }\chi |j|^2
+  \pi \log \eta \sum_{p\in\Lambda} \chi (p) \).
\end{equation}
The limit in the definition exists, as noted in \cite{ss1}.

The ``renormalized energy'' $W_U$ relative to a family of sets
$U = \{U_R\}_{R>0}$  in $\mr^2$ (for example, balls of radius $R$) is then defined
from this  by
\begin{equation}\label{WU} W_U(j)= \limsup_{R \to \infty}
\frac{W(j, \chi_{\UR})}{|\UR|} ,
\end{equation}
where
 $\chi_{\UR}$ denotes  non-negative  cutoff functions satisfying, for some
constant $C$ independent of $R$,
\begin{equation}
\label{defchi} \pnorm{\nab \chi_{\UR}}{\infty} \le C, \quad \supp(\chi_{\UR})
\subset \UR, \quad \chi_{\UR}(x)=1 \ \text{if } d(x, \UR^c) \ge
1.
\end{equation}

This function $W_U$  (we will most generally omit the $U$ subscript) was introduced in \cite{ss1}, where it was derived as a limiting interaction energy for vortices of Ginzburg-Landau configurations (in superconductivity). In this context, it can be viewed as a version of the renormalized energy of Bethuel-Brezis-H\'elein  \cite{bbh}, but  for an infinite number of points and an infinite domain. Independently of Ginzburg-Landau, it can be seen  as a Coulombian interaction energy for an infinite number of points in the plane, computed via a renormalization.  Many of its properties are stated in \cite{ss1}, and we refer the reader to that paper for more details. It is conjectured in \cite{ss1} that the minimum of $W$ is achieved when the set of points $\Lambda$ is a perfect hexagonal lattice in the plane with the suitable density; this corresponds to what is called  the Abrikosov lattice in the context of superconductivity.  In any case, $W$ is expected to measure the order and homogeneity of a point configuration $\Lambda$.

As mentioned in \cite{ss1}, this energy appears beyond the context of Ginzburg-Landau. In particular, in \cite{ss2,ss3}, Sandier and the first author explore the fact that $W$ also arises naturally in the context of  (the statistical mechanics of) log-gases and random matrices. This also led them to provide in \cite{ss3} a definition of a renormalized energy for the (logarithmic) interaction of points on the real line. That one-dimensional version of $W$  is computed by embedding the real line in the plane, changing the constant ``background charge'' from $1$ to $\dr$ -- where $\dr$ denotes the ``Dirac mass'' along the $x_1$-axis of the plane -- and computing the 2D renormalized energy. More precisely, one should replace \eqref{eqj} by
$$\curl j = 2\pi \nu - \dr, \qquad \div j =0 \quad \text{in} \ \mr^2$$
where $\nu =\sum_{p \in \Lambda} \delta_p$ for some discrete set $\Lambda \subset \mr\subset \mr^2$ and $\dr $ is the measure characterized by the fact that for any test function $\phi$,  $ \int_{\Rn{2}} \phi  d\dr  = \int_{\mr} \phi(x_1,0)\, dx_1.$  Then $W(j,\chi)$  and $W(j)$ are  defined through the same formulae \eqref{WR} and \eqref{WU}. In this 1D case, the minimum of $W$ is proven in \cite{ss3} to be  achieved by the perfect one dimensional ``lattice,'' i.e. the set $\frac{1}{2\pi}\mz$.

Here we will give a unified treatment of both cases by considering the more general setting of vector fields satisfying $\curl j=2\pi \nu -m$ and $\div j=0$,  where $m$ is a positive Radon measure that can only charge lines.

The main motivation for the present paper is to obtain optimal estimates that are needed in  \cite{ss2,ss3} for log-gases. Let us explain a bit further the context there. It consists in studying the behavior as $n \to \infty$ of the  probability law
\begin{equation}
\label{loi}
d\Q(x_1, \dotsc, x_n)=   \frac{1}{\Z} e^{-\beta \w (x_1, \dotsc, x_n)}dx_1 \cdots dx_n\end{equation} where $\Z$ is the associated  partition function, i.e. a normalizing factor such that $\Q$ is a probability, and
\begin{equation}\label{wn}
\w (x_1, \dotsc, x_n)= -  \sum_{i \neq j} \log |x_i-x_j| +n  \sum_{i=1}^n V(x_i).
\end{equation}
The points $x_1, \dotsc, x_n$ belong either to the real line (1D log-gases) or to the plane (2D log-gases), and $V$ is some potential with sufficient growth at infinity, typically $V(x)=|x|^2$. In the context of statistical mechanics, the parameter $\beta$ is the inverse of a temperature. The particular cases of $\beta=1,2,4$ with $V$ quadratic also correspond to random matrix models (for more details the reader can consult e.g. \cite{forrester}).

One can observe that $\w$ and $W$ have a similar logarithmic flavor.
In \cite{ss2,ss3}, the main  point is  to explicitly  connect them in the limit $n \to \infty$ and to exploit this connection to deduce estimates on the probability of some events happening. Heuristically, the idea is that $W$ quantifies the order or heterogeneity of a configuration of points in the line or the plane, and that configurations of points with large $W$  have a probability $\Q$  of arising that decays exponentially  as $n \to \infty$.

 To obtain optimal rates on this decay, it turns out that we need to know how $W$ controls $j$ in an
optimal manner. In practice, it suffices to work with the  local version
$W(j, \chi)$, defined by \eqref{WR}, where $\chi$ is a cutoff function. We wish to obtain a control of $j$ by  the   number of points in $\Lambda$, say $n$, via $W(j,\chi)$.  The optimal estimate that we will obtain here, roughly $\|j\|_{L^p} \le C n^{1/p}$ for $1\le p<2$, will be used crucially in \cite{ss2,ss3}.

A weaker control than needed was already established in \cite{ss1},
Lemma 4.6:
\begin{lem}[\cite{ss1}]\label{wlp} Let $\chi$ be a smooth, non-negative function compactly supported in  an open set $U$ of the plane,  and assume that \eqref{eqj} holds in $\widehat U := \{x \mid d(x,U)<1\}$, where $\nu = 2\pi \sum_{p\in\Lambda} \delta_p$ for some finite subset $\Lambda \subset \widehat U$.  Then for any $p\in[1,2)$,
\begin{equation}\label{estss1}
 \int_{U} \chi^{p/2} |j|^p \le C
(|U|+C_p)^{1-p/2}\(W(j,\chi)+ n(\log n+ 1)
\pnorm{\chi}{\infty}+ n\pnorm{\nab\chi}{\infty}\)^{p/2},
\end{equation} where $n = \nu(\widehat U)/2\pi = \#\Lambda$, $C>0$ is a universal constant, and $C_p>0$ a constant depending on $p$.
\end{lem}
Here the number of points in the region $U$ is $n$, and typically
the volume of $U$ is proportional to $n$ and the value of
$W(j,\chi)$ also grows like $n$. The estimate \eqref{estss1} then provides (roughly) the bound $\pnormspace{j}{p}{U} \le C n^{1/p} (\log n)^{1/2}$. This is not optimal; our goal here is to remove the  $(\log n)^{1/2}$ term to obtain the optimal estimate in $n^{1/p}$. This will be achieved by employing a modification of the method we introduced in \cite{st}, which uses the Lorentz space $L^{2,\infty}$ in conjunction with the ``ball construction method'' \`a la  Jerrard \cite{je} and Sandier \cite{sa}.

A definition of the norm in the Lorentz space $\lti$ is
\begin{equation}\label{deflor}
\|f\|_{\lti} = \sup_{|E|<\infty} |E|^{-\hal} \int_E |f(x)|\,
dx.\end{equation}
We will come back to this in Section \ref{seclorentz}.  For more information on Lorentz spaces, we refer to the book \cite{grafakos}.

The reason to use the Lorentz space $\lti$  is as follows.  When $j$ solves \eqref{eqj}, it is equal to $\np H $ for some $H$ that has a logarithmic behavior near each $p
\in \Lambda$ (recall the space dimension is 2). Thus $|j|$ behaves
like $\frac{1}{|x-p|}$ near each $p\in \Lambda$.  This barely fails
to be in $L^2$ (hence the need for the  renormalization in the definition
\eqref{WR}, achieved by cutting out small holes around each $p$); however, it is in the Lorentz space $\lti \supset L^2$, which has the same scaling homogeneity as $L^2$.  We can thus hope for an estimate like $\|j\|_{\lti}\le Cn^{1/2}$, which will yield as corollaries the desired estimates without the $(\log n)^{1/2}$ error in all spaces into which $\lti$ embeds, such as $L^p$ for $1\le p< 2$ (the norms are over sets of finite measure).

\bigskip

Let us now give the complete result we obtain. As already mentioned, we consider open sets $U \subset \Rn{2}$ and vector fields $j: U\to \mr^2$ satisfying
\begin{equation}
\label{curljeq}
\curl  j =2\pi \nu - m, \quad \div j=0  \qquad \text{in } \ U,
\end{equation}
where $\nu = \sum_{p\in\Lambda} \delta_p$ for some finite subset $\Lambda \subset U$, and $m$ is  a positive Radon measure satisfying the following property: there exists $M>0$ such that
\begin{equation}\label{density}\forall 0<r<1, \forall x \in U\quad
m(B(x,r))\le \pi M r.
\end{equation}
Then Theorem 6.9 of \cite{mattila}, for example, implies that $m(A) \le \pi M \mathcal{H}^1(A)$ for every set $A$, where $\mathcal{H}^1$ is the $1-$dimensional Hausdorff measure.  This means that, while $m$ can concentrate, it can only do so on sets of Hausdorff dimension greater than or equal to one.

For any such $j$ and any function $\chi\ge 0$, we define $W(j,\chi)$ according to the formula \eqref{WR}, where the limit still exists.  Our main result is

\begin{theo}\label{main}
Let $\chi$ be a smooth, non-negative function  compactly supported in  an open set $U \subset \Rn{2}$,  and assume that $\curl j = 2\pi \nu - m $, $\div j=0$ in $\widehat U := \{x\mid d(x,U)<1\}$, where $\nu = \sum_{p\in\Lambda} \delta_p$ for some finite subset $\Lambda$ of $\widehat U$, and $m$ is a positive Radon measure satisfying \eqref{density}.  Then there exists an explicitly constructed
vector field  $G$ in $\widehat U$ satisfying the following.
\begin{enumerate}
\item $\|G\|_{\lti(\widehat U)}^2 \le C n$, where $C$ is universal,
\item  for any $\beta>0$, there exists $C_\beta>0$ depending only on $\beta$ and $M$, such that
\begin{equation}
\label{jW} \hal \int_U \chi  |j- G|^2 \le  (1+\beta) W(j, \chi) + C_\beta
n( \pnorm{\chi}{\infty} + \pnorm{\nab \chi}{\infty}) + C_\beta n' \log n'
\end{equation}where $n = \nu(\widehat U) = \#\Lambda$, and
$n'= \#\{p\in\Lambda\mid B(p,\hal)\cap \{0< \chi \le \hal \pnorm{\chi}{\infty} \} \neq\varnothing\}$.
\end{enumerate}
\end{theo}

The purpose of coupling the $L^{2,\infty}$ estimate of $G$ to the $L^2$ estimate of $j-G$ is to allow $G$ to be eliminated from the estimate via the triangle inequality, resulting in an $L^{2,\infty}$ estimate for  $j$ alone.  This then yields an estimate in $L^p$ since $L^{2,\infty}  \hookrightarrow L^p$ on sets of finite measure. Taking $\beta=1$, for example, we can obtain the following.

\begin{coro}\label{main_cor}
 Under the same assumptions,  for every  $1\le p<2$ we have
\begin{equation*}
\frac{1}{C_p |U|^{\frac{1}{p}- \hal}} \pnormspace{\sqrt{\chi} j}{p}{U} \le \pqnormspace{\sqrt{\chi} j}{2}{\infty}{U} \le C    \left( W (j, \chi) + n ( \pnorm{\chi}{\infty} + \pnorm{\nab \chi}{\infty}) + n' \log n'\right)^{\hal},
\end{equation*}
where $n$ and $n'$ are as in the theorem, $C>0$ is a universal constant, and $C_p>0$ is a constant depending on $p$.
\end{coro}

When $\chi$ is a cutoff function associated to the domain $U$, the
 term $n'\log n'$ is  a boundary contribution that we typically expect to be negligible relative to $n$.  For example, if for balls of radius $R$, $n$ scales like $n \sim \pi R^2$, then $n'$ can be regarded as the number of elements  of  $\Lambda$ in the annulus $B(0,R)\backslash B(0,R-1)$.  Then
\begin{equation*}
 n' \sim 2\pi R  \Rightarrow n' \log n'  \ll  n
\end{equation*}
as $R \to \infty$. Moreover, $|\widehat U|\sim \pi R^2\sim n$ and  as we mentioned, we expect $W$ to be typically of order $n$. The result of this  corollary in such a situation  is then that
\begin{equation*}
\pqnormspace{\sqrt{\chi} j}{2}{\infty}{U} \le C n^{1/2} \text{ and } \pnormspace{\sqrt{\chi} j}{p}{U} \le C n^{1/p}.
\end{equation*}
The $L^p$ estimate should be compared to Lemma \ref{wlp}: we improve from $n^{1/p} (\log n)^{1/2}$ to the optimal power $n^{1/p}$.

\bigskip

The paper is organized as follows. In Section \ref{sec2} we recall the various definitions that we need for $\lti$ and we see how to estimate $\lti$ norms for vector fields defined on non-overlapping annuli.  In Section \ref{sec3} we return to the ball construction for \eqref{eqj}, borrowed from  Section 4 of \cite{ss1}.  We improve the estimates it yields by utilizing methods we introduced in \cite{st}.  More specifically, we construct a vector field $G$ that mimics the optimal behavior around each $p\in \Lambda$, and then we plug in the explicit Lorentz estimates of the previous section.

\begin{remark}\label{notation}
Throughout the paper, $B(x,r)$ denotes the open ball of center $x$ and radius $r$, while $\bar{B}(x,r)$ denotes the closed ball of center $x$ and radius $r$.  If $B=B(x,r)$, then for any $\lambda >0$ we write $\lambda B$ for $B(x,\lambda r)$.  We employ the same notation for closed balls.
\end{remark}

%%%%%%%%%%%%%%%%%%%%%%%%%%%%%%%%%%%%%%%%%%%%%%%%%%%%%%%%%%%%%%%%%
\section{Lorentz space estimates} \label{sec2}
%%%%%%%%%%%%%%%%%%%%%%%%%%%%%%%%%%%%%%%%%%%%%%%%%%%%%%%%%%%%%%%%%

%%%%%%%%%%%%%%%%%%%%%%%%%%%%%%%%%%%%%%%%%%%%%%%%%%%%%%%%%%%%%%%%%
\subsection{Definition and properties of the Lorentz space $L^{2,\infty}$} \label{seclorentz}
%%%%%%%%%%%%%%%%%%%%%%%%%%%%%%%%%%%%%%%%%%%%%%%%%%%%%%%%%%%%%%%%%

In this subsection we start by recalling the definition of the Lorentz space $L^{2,\infty}$ and the properties we will need.  Let $\Omega \subseteq \Rn{2}$.  For a function $f:\Omega \rightarrow
\Rn{k}$, $k\ge 1$, we define the distribution function of $f$ by
\begin{equation}
\lambda_f(t) = \abs{\{x \in \Omega \;|\; \abs{f(x)} > t   \}},
\end{equation}
where  $\abs{A}$ denotes the Lebesgue measure of the set $A$.  We then define the quasi-norm
\begin{equation}\label{quasi_def}
\wnorm{f} = \sqrt{\sup_{t>0}t^2\lambda_f(t)},
\end{equation}
and the Lorentz space $L^{2,\infty}(\Omega) = \{f \;|\; \wnorm{f} < \infty \}$.  As a quasi-norm, the quantity $\wnorm{\cdot}$ behaves as a norm except in the triangle inequality, where it instead satisfies $\wnorm{f + g} \le C( \wnorm{f} + \wnorm{g})$ for  some $C > 1$.   The space $L^{2,\infty}$, as defined by the quasi-norm, is only a quasi-Banach space, i.e. a linear space in which every quasi-norm Cauchy sequence converges in the quasi-norm.   However, it can be normed by introducing the norm
\begin{equation}
  \pqnorm{f}{2}{\infty}  = \sup_{\abs{E}<\infty} \abs{E}^{-1/2} \int_{E} \abs{f(x)}dx.
\end{equation}
We then have that (see, for example,  Lemma 6.1 of \cite{st})
\begin{equation}\label{normed}
\wnorm{f} \le \pqnorm{f}{2}{\infty} \le 2 \wnorm{f}.
\end{equation}
For a more thorough discussion of $L^{2,\infty}$, and the other Lorentz spaces in general, we refer to the book \cite{grafakos}.

It is perhaps more natural to work with the norm rather than the quasi-norm.  However, the estimates we need are easier  to derive with the quasi-norm, so we will mostly work with it.  We now record a Lemma on some properties of the quasi-norm $\wnorm{\cdot}$.  The proof follows directly from the definition \eqref{quasi_def}, and is thus omitted.

\begin{lem}\label{quasi_properties}
The quasi-norm $\wnorm{\cdot}$ satisfies the following properties.
\begin{enumerate}
 \item If $\abs{f(x)} \le \abs{g(x)}$ for a.e. $x$, then $\wnorm{f} \le \wnorm{g}$.

 \item Suppose $f = f_1 + f_2$ with $\supp(f_1) \cap \supp(f_2) = \varnothing$.  Let $T_1,T_2$ be translation operators so that $\supp(T_1 f_1) \cap \supp(T_2 f_2)= \varnothing$.  Then
\begin{equation}
 \wnorm{f} = \wnorm{T_1 f_1 + T_2 f_2}.
\end{equation}

 \item   If $f = f_1 + f_2$ with $\supp(f_1) \cap \supp(f_2) = \varnothing$, then
\begin{equation}
 \wnorm{f}^2 \le \wnorm{f_1  }^2 + \wnorm{f_2  }^2.
\end{equation}

\item Let $f(x) = 1/\abs{x-c}$ for some $c \in \Rn{2}$.  Then $\wnorm{f} = \sqrt{\pi}$.

\end{enumerate}

\end{lem}

%%%%%%%%%%%%%%%%%%%%%%%%%%%%%%%%%%%%%%%%%%%%%%%%%%%%%%%%%%%%%%%%%
\subsection{Minimal concentric rearrangement}
%%%%%%%%%%%%%%%%%%%%%%%%%%%%%%%%%%%%%%%%%%%%%%%%%%%%%%%%%%%%%%%%%

In this subsection, we define the notion of  minimal concentric rearrangement number for a finite collection of annuli.  This number, which we had not previously introduced in \cite{st}, will serve as a tool for estimating the Lorentz space norms of vector fields defined on annuli obtained from the ``ball construction.''

Consider a finite collection of annuli, $\mathcal{A} = \{ A_i
\}_{i=1}^M$, where
\begin{equation}\label{annuli_form}
A_i = \{ x \in \Rn{2} \;\vert\;  r_i < \abs{x-c_i} \le s_i  \}
\text{ for } c_i \in \Rn{2}, 0 < r_i < s_i < \infty.
\end{equation}
We say that such a collection $\mathcal{A}$ may be concentrically
rearranged if the annuli of $\mathcal{A}$ can be translated in
$\Rn{2}$ so that the translates share a common center and are
pair-wise disjoint.  This is obviously equivalent to the property
that, up to relabeling the indices, the inner-radii and outer-radii
satisfy
\begin{equation}
  r_1 < s_1 \le r_2 < s_2 \le  \cdots \le r_{M-1}< s_{M-1} \le r_{M} < s_M.
\end{equation}

Clearly, not every finite collection of annuli can be concentrically
rearranged.  However, every such $\mathcal{A}$ can be partitioned
into disjoint subcollections $\{ \mathcal{A}_k\}_{k=1}^K$ such that
$\mathcal{A} = \cup_{k=1}^K \mathcal{A}_k$ and each $\mathcal{A}_k$
can be concentrically rearranged.  This property trivially holds,
for instance, if $\mathcal{A}_i = \{A_i\}$ for $i=1,\dotsc,M=K$.  In
general, though, this trivial partitioning into singletons is not
optimal in terms of $K$.  We pursue this optimal $K \in
\{1,\dotsc,M\}$ via the following definition.

\begin{definition}\label{mcr_def}
Let $\mathcal{A}$ be a finite collection of annuli.  The ``minimal
concentric rearrangement'' number of $\mathcal{A}$ is defined as
\begin{multline}
 \mcr(\mathcal{A}) = \min \{ K \in \mathbb{N} \;\vert\;  \mathcal{A} = \cup_{k=1}^K \mathcal{A}_k \text{ with } \mathcal{A}_i \cap \mathcal{A}_j = \varnothing \text{ for } i \neq j
 \text{ so that }\\
\mathcal{A}_k  \text{ can be concentrically
rearranged for }k=1,\dots,K \}.
\end{multline}
\end{definition}

\begin{remark}
If $\mcr(\mathcal{A}) =K$ and $\mathcal{A} = \cup_{k=1}^K \mathcal{A}_k$, then it must hold that $\mathcal{A}_k \neq \varnothing$ since otherwise empty sets may be removed from the partition of $\mathcal{A}$, contradicting the definition of $\mcr(\mathcal{A})$.
\end{remark}

\begin{remark}\label{mcr_rem}
If $\mathcal{A}_1$ and $\mathcal{A}_2$ are two finite collections of
annuli, then $\mcr(\mathcal{A}_1 \cup  \mathcal{A}_2) \le
\mcr(\mathcal{A}_1) + \mcr(\mathcal{A}_2)$.
\end{remark}

We will be interested, in particular, in collections of annuli that come from ball-growth procedures, which were first introduced in \cite{je,sa}.  To define the ball-growth procedure, we record the following result, which is  Theorem 4.2 of \cite{ssbook}, except that here we have ``reparameterized'' the ball-growth parameter.  Recall that we use the notational conventions mentioned in Remark \ref{notation}.

\begin{lem}\label{ball_growth}
Let $\mathcal{B}_0$ be a finite, disjoint collection of closed balls so that the total radius of $\mathcal{B}_0$ is $r_0>0,$ i.e. $\sum_{B \in \B_0} r(B) = r_0$, where $r(B)$ denotes the radius of the ball $B$.  Let $r> r_0$.  Then there exists a family $\{\mathcal{B}(t)\}_{t\in [r_0,r]}$ of collections of disjoint, closed balls such that the following hold.
\begin{enumerate}
 \item $\mathcal{B}(r_0)=\mathcal{B}_0$.
 \item For $r_0 \le t \le s \le r$,
\begin{equation}
 \bigcup_{B \in \mathcal{B}(t)} B \subseteq \bigcup_{B \in \mathcal{B}(s)} B.
\end{equation}
 \item There exists a finite set $S \subset (r_0,r]$ such that if $[t,s] \subset [r_0,r]\backslash S$, then $\mathcal{B}(s) = \frac{s}{t} \mathcal{B}(t)$.  In particular, if $B(s) \in \mathcal{B}(s)$ and $B(t) \in \mathcal{B}(t)$ are such that $B(t) \subset B(s)$, then $B(s) = \frac{s}{t} B(t)$.  The set $S$ is referred to as the set of ``merging times.''
 \item For every $t \in [r_0,r]$, the total radius of $\B(t)$ is $t$, i.e.
\begin{equation}\label{bg_1}
 \sum_{B \in \mathcal{B}(t)} r(B) =  t.
\end{equation}
\end{enumerate}
\end{lem}

The proof of this lemma proceeds roughly as follows. The initial finite set of disjoint, closed balls has their radii grown, all at the same multiplicative rate, until two (or more) grown balls become tangent.  At this time (an element of the set $S$), the tangent balls are ``merged'' into a larger ball in such a way that the sums of the radii are preserved.  Then the growth procedure is started again.  The resulting family $\{\mathcal{B}(t)\}_{t\in [r_0,r]}$ can be thought of as a ``piece-wise continuous'' growth process with ``jump discontinuities'' at the merging times $S$.  However, as guaranteed by \eqref{bg_1}, the sum of the radii of the balls is continuous.

\begin{remark}\label{ball_growth_remark}
Let $\B_0$ be a finite, disjoint collection of closed balls of total radius $r_0>0$, and let $\{\B(t)\}_{t\in[r_0,r]}$ be the family generated by the ball-growth procedure of Lemma \ref{ball_growth}, where $r_0 < r < \infty$.  The collection $\B(r)$ may then be used as the starting point of another ball-growth procedure to generate $\{\B(t)\}_{t\in[r,r_1]}$.  We may then concatenate the two families to form $\{\B(t)\}_{t\in[r_0,r_1]}$.  This family satisfies all of the conclusions of Lemma \ref{ball_growth} with $r$ replaced by $r_1$, and we may view the new family as having been generated by a single growth procedure.  In other words, if the final collection of a ball-growth coincides with the starting collection of another growth, we can join the two and view what results as a single growth process.
\end{remark}

Given a family $\{\mathcal{B}(t)\}_{t\in [r_0,r]}$ generated by the ball-growth procedure, we wish to define a corresponding finite collection of disjoint annuli of the form \eqref{annuli_form} -- these are simply the annuli generated through the ball-growth. We do so now in the following definition.

\begin{definition}\label{b_growth_annuli}
Suppose that a finite, disjoint collection of closed balls
$\B_0$ is grown via a ball-growth procedure according to Lemma \ref{ball_growth} into $\{\mathcal{B}(t)\}_{t \in [r_0,r]}$.  Let $S \subset (r_0,r]$ be the
finite set of merging times, and write $N = \#(S)$.  If $N =0$ we
define $t_0 =r_0$ and $t_1 = t_{N+1} = r$.  If $N \ge 1$, then we
enumerate $S = \{t_i\}_{i=1}^N$ so that $r_0 < t_1 < t_2 < \cdots <
t_N \le r$, and then we set $t_0 = r_0$ and $t_{N+1}=r$.  We then
define the finite collection of disjoint annuli $\mathcal{A}$ according to
\begin{equation}
 \mathcal{A} = \bigcup_{i=0}^N \bigcup_{B \in \mathcal{B}(t_i)} \left(\frac{t_{i+1}}{t_{i}} B\right) \backslash B.
\end{equation}
Note that if $t_{i+1} > t_{i}$, then $((t_{i+1}/t_{i}) B)
\backslash B$ is an annulus of the form written in
\eqref{annuli_form}.  We say that the collection  $\mathcal{A}$ is
generated by a ball-growth, starting from $\B_0$.

\end{definition}

If the finite collection of annuli $\mathcal{A}$ is generated by a
ball-growth, then it is possible to estimate $\mcr(\mathcal{A})$ in
terms of the number of initial balls in the ball-growth.  This
estimate is the content of our next result.

\begin{pro}\label{ball_growth_mcr}
Suppose that $\mathcal{A}$ is a finite collection of annuli
generated by a ball-growth, starting from a disjoint collection of $n \ge 1$
closed balls.  Then $\mcr(\mathcal{A}) \le n$.
\end{pro}

\begin{proof}

It suffices to prove the result assuming that the final collection
of balls generated  by the ball-growth procedure is just a single
ball.  In the general case with $m$ final balls, say
$\{B_i\}_{i=1}^m$, each $B_i$ can  be viewed as having been grown
from $n_i \ge 1$ disjoint balls, where $\sum_{i=1}^m n_i = n$.
Writing  $\mathcal{A} = \cup_{i=1}^m \mathcal{A}_i$ for
$\mathcal{A}_i$ consisting of the annuli contained in $B_i$, we may
use the single-final-ball estimate $\mcr(\mathcal{A}_i) \le n_i$ in
conjunction with Remark \ref{mcr_rem} to estimate
\begin{equation}
 \mcr(\mathcal{A}) = \mcr\left(\bigcup_{i=1}^m \mathcal{A}_i \right) \le \sum_{i=1}^m \mcr(\mathcal{A}_i) \le \sum_{i=1}^m n_i = n.
\end{equation}
We will thus restrict to proving the result when the final
collection is just a single ball.  The proof proceeds by induction
on the number of initial balls, $n \ge 1$.

In the case $n=1$ there is a single initial and final ball, so
$\mathcal{A} = \{ A\}$ for $A = \{  r_0 < \abs{x-c} \le r\}$ for some
$c \in \Rn{2}$ and $r > r_0$.  This collection is a single
annulus, and is thus trivially concentrically arranged.  Hence
$\mcr(\mathcal{A}) = 1 =n$.

Suppose now that for all $1 \le k \le n$, $\mcr(\mathcal{A}) \le k$
for any finite collection of annuli generated by a ball-growth,
starting from $k$ initial disjoint closed balls and ending in a
single final ball.  Let $\mathcal{A}$ be a finite collection of
annuli generated by a ball-growth, starting from $n+1$ disjoint
closed balls and ending in a single final ball.  We will show that
$\mcr(\mathcal{A}) \le n+1$, which then proves the desired result
for arbitrary $n$ by induction.

Let $\{ \mathcal{B}(t)\}_{t \in [r_0,r]}$ be the family of collections
of balls generated in the process, with $\mathcal{B}(r) = \bar{B}(c,r) $ for some $c \in \Rn{2}$ and $r>r_0$.  Since $n+1 \ge 2$, the ball-growth procedure that generated $\mathcal{A}$ must
have involved at least one merging time.  Let us call the largest
merging time  $T_* \in (r_0,r]$.  The ball-growth procedure dictates
that at $T_*$ a collection of closed balls $\{B_i(T_*) \}_{i=1}^J$,
$J \ge 2$,  merged into a single ball $\bar{B}(c,s)$; if $T_* = T$  then
$r=s$, and if $T_* < T$ then $ s < r$.  We will assume that $T_* <
T$.  The case $T = T_*$ is easier and may be handled with an
obvious variant of the argument below.

Since $T_* < T$, one of the annuli in $\mathcal{A}$ is $\bar{B}(c,r)
\backslash \bar{B}(c,s)$.  All of the remaining annuli in $\mathcal{A}$
are subsets of exactly one of the balls in  $\{B_i(T_*) \}_{i=1}^J$.
This allows us to write $\mathcal{A}$ as a disjoint union:
\begin{equation}\label{bgm_1}
 \mathcal{A} = \{\bar{B}(c,r) \backslash \bar{B}(c,s)\} \cup \bigcup_{i=1}^J \mathcal{A}_i,
\end{equation}
where $\mathcal{A}_i = \{ A \in \mathcal{A} \;\vert\; A \subset
B_i(T_*) \}$.  Let $ n_i$ denote the number of initial balls
contained in $B_i(T_*)$.  Clearly, $\sum_{i=1}^J n_i =n+1$.   Also, $1\le
n_i \le n$ since if one of the $n_i$'s were equal to $n+1$, then
$J$ would have to be $1$, in contradiction with the fact that there
is a merging at $T_*$. Then each $\mathcal{A}_i$ can be regarded as
having been generated by a ball-growth, starting from $n_i$ disjoint
closed balls and ending in a single final ball, $B_i(T_*)$. The
induction hypothesis implies that $\mcr(\mathcal{A}_i) \le n_i$ for
$i=1,\dotsc,J$, which we may combine with \eqref{bgm_1} and Remark
\ref{mcr_rem} to see that
\begin{equation}\label{bgm_2}
\begin{split}
\mcr(\mathcal{A}) & \le \mcr( \{\bar{B}(c,r) \backslash \bar{B}(c,s)\}
  \cup  \mathcal{A}_1) + \sum_{i=2}^J \mcr(\mathcal{A}_i) \\
& \le \mcr( \{\bar{B}(c,r) \backslash \bar{B}(c,s)\}
  \cup  \mathcal{A}_1) + \sum_{i=2}^J n_i \\
&=  \mcr( \{\bar{B}(c,r) \backslash \bar{B}(c,s)\}
  \cup  \mathcal{A}_1)  + (n+1 - n_1).
\end{split}
\end{equation}

We claim that
\begin{equation}\label{bgm_3}
\mcr( \{\bar{B}(c,r) \backslash \bar{B}(c,s)\} \cup  \mathcal{A}_1)  \le n_1.
\end{equation}
Once this is established, \eqref{bgm_2} and \eqref{bgm_3} imply that
$\mcr(\mathcal{A}) \le n+1$, which completes the proof.  To prove
the claim, we first note that by \eqref{bg_1} of Lemma \ref{ball_growth}, the ball-growth procedure requires that if $r_i$ is the radius of $B_i(T_*)$, then $\sum_{i=1}^J r_i =
s$, which in particular means that $r_1 < s$.  This and the fact
that $\mcr(\mathcal{A}_1) \le n_1$ imply that $\mathcal{A}_1$ can be
rearranged into at most $n_1$ disjoint collections of concentric
annuli, each of which can be contained in a concentric ball of
radius $r_1 < s$.  We may then translate one of these collections of
disjoint, concentric annuli to have center $c$ so that the union of
these translated annuli with  $\bar{B}(c,r)\backslash \bar{B}(c,s)$  forms a
single new collection of  disjoint, concentric annuli.   From this we
easily deduce that \eqref{bgm_3} holds, finishing the proof.

\end{proof}

%%%%%%%%%%%%%%%%%%%%%%%%%%%%%%%%%%%%%%%%%%%%%%%%%%%%%%%%%%%%%%%%%
\subsection{Lorentz space estimates for vector fields on annuli}
%%%%%%%%%%%%%%%%%%%%%%%%%%%%%%%%%%%%%%%%%%%%%%%%%%%%%%%%%%%%%%%%%
Now we provide an estimate of the $L^{2,\infty}$ quasi-norm of
certain vector fields that are supported on disjoint annuli.  Our present estimate is somewhat easier than a similar estimate we proved in \cite{st}.  The reason for this is that we are now interested in estimates in terms of the number of initial balls (related to the minimal concentric rearrangement number of the annuli through Proposition \ref{ball_growth_mcr}), but in \cite{st} we were (roughly speaking) concerned with estimates in terms of the number of final balls.

We now turn to our estimate of the $L^{2,\infty}$ quasi-norm in terms of $\mcr(\mathcal{A})$.

\begin{pro}\label{norm_mcr}
 Suppose that $\mathcal{A} = \{A_i\}_{i=1}^M$ is a finite, disjoint collection of annuli of the form \eqref{annuli_form}, with centers $c_i$.  Let
\begin{equation}
 f(x) = \sum_{i=1}^M \vchi_{A_i}(x) \frac{v_i(x)}{\abs{x-c_i}}
\end{equation}
for vector fields $v_i$ satisfying the bound $\abs{v_i(x)} \le \alpha < \infty$ on $A_i$ for
$i=1,\dotsc,M$.  Then
\begin{equation}
 \wnorm{f} \le \alpha \sqrt{\pi \mcr(\mathcal{A}}).
\end{equation}
\end{pro}

\begin{proof}
Write $\mcr(\mathcal{A}) = K \in \{1,\dotsc,M\}$ and let $\{ \mathcal{A}_k \}_{k=1}^K$ satisfy  $\mathcal{A} = \cup_{k=1}^K \mathcal{A}_k$   with $\mathcal{A}_i \cap \mathcal{A}_j = \varnothing$ for $i \neq j$ and so that
  each $\mathcal{A}_k \neq \varnothing$  can be concentrically rearranged.  We are free to enumerate $\mathcal{A}_k = \{A_{k,j}\}_{j=1}^{N_k}$ so that
\begin{equation}
 r_{k,1} < s_{k,1} \le r_{k,2} \le s_{k,2} \le \cdots \le r_{k,N_k} \le s_{k,N_k},
\end{equation}
where $r_{k,j}$ and $s_{k,j}$ denote the inner and outer radii
(respectively) of $A_{k,j}$ for $k=1,\dotsc,K$ and $j=1,\dotsc,N_k$.

By performing the concentric rearrangements and employing the second
and third properties of Lemma \ref{quasi_properties}, we see that
\begin{equation}
 \wnorm{f}^2 = \wnorm{\sum_{k=1}^K \sum_{j=1}^{N_k} g_{k,j} }^2 \le \sum_{k=1}^K  \wnorm{\sum_{j=1}^{N_k} g_{k,j} }^2,
\end{equation}
where
\begin{equation}
 g_{k,j}(x) = \vchi_{\{r_{k,j} <\abs{x-c_k}  \le  s_{k,j}\}}(x) \frac{v_{k,j}(x)}{\abs{x-c_k}}
\end{equation}
for points $c_k \in \Rn{2}$ and vector fields with $\abs{v_{k,j}(x)}
\le \alpha < \infty$ (translations of the $v_k$).  By construction, we have that
for each $k$,
\begin{equation}
 \abs{\sum_{j=1}^{N_k} g_{k,j}(x)} \le \frac{\alpha}{\abs{x-c_k}} \text{ for all } x \in \bar{B}(c_k,s_{k,N_k}).
\end{equation}
Then, according to the first and fourth properties of Lemma
\ref{quasi_properties}, for each $k=1,\dotsc,N$ we have that
\begin{equation}
  \wnorm{\sum_{j=1}^{N_k} g_{k,j} }^2 \le \alpha^2 \wnorm{ \abs{\cdot-c_k}^{-1} }^2 \le \alpha^2 \pi.
\end{equation}
Hence,
\begin{equation}
 \wnorm{f}^2 \le \sum_{k=1}^K  \wnorm{\sum_{j=1}^{N_k} g_{k,j} }^2 \le \alpha^2 \sum_{k=1}^K \pi = \alpha^2 \pi K.
\end{equation}

\end{proof}

As a direct corollary, we obtain the main result of this section.

\begin{pro}\label{lorentz_est}
Suppose that $\mathcal{A} = \{A_i\}_{i=1}^M$ is a finite, disjoint collection
of annuli of the form \eqref{annuli_form}, generated by a ball-growth, starting from a disjoint
collection of  $n$ closed balls, as defined in Definition \ref{b_growth_annuli}.  Let $c_i\in \Rn{2}$ denote the center of $A_i$. Let
\begin{equation}\label{form}
 f(x) = \sum_{i=1}^M \vchi_{A_i}(x) \frac{v_i(x)}{\abs{x-c_i}}
\end{equation}
for vector fields $v_i$ satisfying the bounds $\abs{v_i(x)} \le \alpha < \infty$ on $A_i$ for
$i=1,\dotsc,M$.  Then
\begin{equation}\label{le_0}
 \pqnorm{f}{2}{\infty} \le 2\alpha  \sqrt{\pi n}.
\end{equation}
\end{pro}
\begin{proof}
Propositions \ref{ball_growth_mcr} and \ref{norm_mcr} imply that
$\wnorm{f} \le \alpha \sqrt{\pi n}$.  The estimate \eqref{le_0} follows
from this and the  estimate \eqref{normed}.
\end{proof}

%%%%%%%%%%%%%%%%%%%%%%%%%%%%%%%%%%%%%%%%%%%%%%%%%%%%%%%%5
\section{Improved ball construction estimates}\label{sec3}
%%%%%%%%%%%%%%%%%%%%%%%%%%%%%%%%%%%%%%%%%%%%%%%%%%%%%

In this section, we return to the ball construction \`a la Jerrard and Sandier \cite{je,sa} that was introduced in the context of \eqref{eqj} in \cite{ss1}, Section 4. We incorporate a term in $G$ as in \cite{st}.

\begin{pro}\label{boulesren} Assume that \eqref{curljeq} holds in the sense of distributions in some open set $U$,  where $m$ satisfies \eqref{density}.  Further assume that $j\in L^2_\loc(U\sm\Lambda)$.  Write $n = \#\Lambda$ and
\begin{equation*}
 \eta_0 =  \hal \min\{ \abs{p-q} \;\vert\; p,q \in \Lambda, p\neq q  \}>0.
\end{equation*}
There exists a family of finite collections of disjoint closed balls $\{\B_r\}_{r\in(0,1]}$ and a vector field $G:\cup_{B \in \mathcal{B}_1} B \to \Rn{2}$ such that the following hold.
\begin{enumerate}

\item For each $r \in (0,1]$, $\mathcal{B}_r$ covers $\Lambda$ and has total radius $r$, i.e. $r = \sum_{B \in \mathcal{B}_r} r(B)$, where $r(B)$ denotes the radius of the ball $B$.  The set $\cup_{B \in \mathcal{B}_r} B$ is increasing as a function of $r$.  Moreover, if $r \le  n \eta_0$, then $\B_r = \{\bar{B}(p,\frac{r}{n})\}_{p\in\Lambda}$.

\item For any $0< \eta < \min\{\eta_0,r/n\}$, $\mathcal{B}_r$ may be viewed as having been generated by the ball-growth procedure of Lemma \ref{ball_growth}, starting from the initial collection  $\mathcal{B}_0 = \{ \bar{B}(p,\eta)\}_{p \in \Lambda}$.

\item For any $0< \eta < \min\{\eta_0,r/n\}$, let $\mathcal{A}(\eta,r)$ denote the finite, disjoint collection of annuli generated from $\{\mathcal{B}(t)\}_{t \in [n \eta,r]}$ according to Definition \ref{b_growth_annuli}.  Then, when restricted to the set $\cup_{B \in \mathcal{B}_r} B$, the vector field $G$ is
\begin{equation*}
 G(x) = \sum_{A \in \mathcal{A}(\eta,r)} \vchi_{A}(x) \frac{(x-c_A)^\bot}{\abs{x-c_A}^2} + \sum_{p \in \Lambda} \vchi_{\bar{B}(p,\eta)}(x)\frac{(x-p)^\bot}{\abs{x-p}^2},
\end{equation*}
where $c_A \in \Rn{2}$ is the center of the annulus $A$ and $x^\bot = (x_2,-x_1)$ for $x \in \Rn{2}$.

\item  For every $0< \eta < \min\{\eta_0,r/n\}$ and every $B \in \B_r$ such that $B \subset U$, we have  (writing $n_B= \#(\Lambda\cap B)$)
\begin{equation}\label{br_01}
\hal \int_{B\sm\cup_{p\in\Lambda}\bar{B}(p,\eta)} |j|^2\ge \pi n_B \left( \log\frac r{n\eta} - M r \right) +   \hal \int_{B\sm\cup_{p\in\Lambda}\bar{B}(p,\eta)} |j- G|^2.
\end{equation}

\item For any $\beta>0$ there exists $C_\beta>0$ such that the following holds: if $0< \eta < \min\{\eta_0,r/n\}$, $B\in\B_r$,  and  $\chi$ is a non-negative function with support in  $B\cap U$, then
\begin{multline}\label{br_02}
\int_{B\sm\cup_{p\in\Lambda}\bar{B}(p,\eta)} \chi |j|^2 - 2\pi\(\log\frac r{n\eta} - M r\)\sum_{p\in B\cap\Lambda} \chi(p)\\ \ge   \frac{1}{1+\beta} \int_{B\sm\cup_{p\in\Lambda}\bar{B}(p,\eta)} \chi |j- G|^2 - C_\beta r \nu(B) \pnorm{\nab\chi}{\infty}.
\end{multline}
\end{enumerate}
\end{pro}

\begin{proof}

The proof is an adaptation of Proposition 4.5 of \cite{ss1}, improved as in \cite{st}.  We proceed through several steps.

{\it Step 1:}  Ball growth

In order to define $\mathcal{B}_r$, we first fix a reference family of balls produced via a ball-growth.  Set $\eta_1 = \min\{\eta_0/2,1/(n+1)\}$ and let $\mathcal{B}_0  = \{ \bar{B}(p,\eta_1)\}_{p \in \Lambda}$.  According to the definition of $\eta_0$, we have that $\mathcal{B}_0$ is a finite, disjoint collection of closed balls of total radius $n \eta_1 < 1$. We apply Lemma \ref{ball_growth} to $\mathcal{B}_0$  to produce the family of collections $\{ \mathcal{B}(t) \}_{t \in [n \eta_1,1]}$, satisfying the conclusions of the lemma.

Now we extend this reference family ``backward'' to radii smaller than $n \eta_1$.  For any $0 < r_0 \le t \le n \eta_1$ we write  $\B(t) = \{\bar{B}(p,t)\}_{p\in \Lambda}$.  Then since the balls in these collections never become tangent, we may trivially view $\{\B(t)\}_{t\in [r_0,n \eta_1]}$ as having been generated by a ball-growth, i.e. all of the conclusions of Lemma \ref{ball_growth} apply to this family.

According to Remark \ref{ball_growth_remark}, we may then combine our reference family with the new one to produce $\{\B(t)\}_{t \in [r_0, 1]}$ for any $0 < r_0 <1$.    We now set $\B_r = \B(r)$ by choosing any $0<r_0< r \le 1$.  This proves the first item.  The second item follows by taking $r_0 = n\eta$.

{\it Step 2:}  Defining $G$

Suppose that $0< \eta < \min\{\eta_0,r/n\}$ and let $\mathcal{A}(\eta,r) = \{ A_i\}_{i=1}^M$ be the collection of disjoint annuli of the form \eqref{annuli_form} generated from $\{ \mathcal{B}(t) \}_{t \in [n \eta,r]}$ according to Definition \ref{b_growth_annuli}.  Let $c_i \in \Rn{2}$ denote the center of $A_i$ and define $v_i(x) = (x-c_i)^\bot/\abs{x-c_i}$.  Note that $\abs{v_i(x)}=1$ for $x \in A_i$.  Now we define the restriction of $G(x)$ to the set $\left(\cup_{B \in \mathcal{B}_r} B \right) \backslash \left(\cup_{p \in \Lambda} \bar{B}(p,\eta)\right)$ as the right side of \eqref{form} with this choice of fields $v_i$.  This definition of $G$ is clearly independent of $\eta$ and $r$ in the sense that for different choices of $\eta$ and $r$, the corresponding $G$ vector fields agree on the set where they are both defined.  We may then unambiguously define $G:  \cup_{B \in \mathcal{B}_1} B \to \Rn{2}$ by sending $\eta \to 0$ and then $r \to 1$.    This proves the third item.

{\it Step 3:} Introducing $G$ in the estimates

Since $\curl j = 2\pi \nu -m$,  for any circle $C = \partial B$ of radius $r_B$ not intersecting $\Lambda$, we have,  letting $d_B =  \#(\Lambda\cap B)$ and $\tau$ denote the oriented unit tangent to $C$,
\begin{equation}\label{circle1}
\int_{C} j\cdot\tau =   2\pi \nu(B) - m(B)  = 2\pi d_B- m(B).
\end{equation}
Suppose now that $C \subset A_i\in \mathcal{A}(\eta, r)$ for one of the annuli constructed in the previous step, with $C$ and $A_i$ centered at the same point $a$.  Then by construction, $G(x)=\frac{ \tau(x)}{\abs{x-a}}$  for $x \in C$, where  $\tau(x)$ the unit tangent at $x \in C$.  We then have, in view of \eqref{circle1}, that
\begin{eqnarray}
\nonumber \int_{C}|j-G|^2 & = & \int_{C} |j|^2 + \int_{C}|G|^2 - \frac{2}{r_B}\int_C j \cdot \tau\\
 \nonumber & = & \int_{C} |j|^2 + 2\pi \frac{1}{r_B}  -
  \frac{2}{r_B}(2\pi d_B - m(B) )\\
  \nonumber & = & \int_{C} |j|^2  - 2\pi \frac{2 d_B-1 }{r_B}  + \frac{2  m(B)} {r_B} \\
  \label{circle3}
  & \le  &  \int_{C} |j|^2  - 2\pi \frac{d_B}{r_B}+ \frac{2  m(B)} {r_B}  ,
  \end{eqnarray}
where we have used the fact that $2 d_B -1 \ge d_B$ since $d_B$ is a positive integer. We thus deduce with \eqref{density} that
\begin{equation}\label{circle4}
\int_{C} |j|^2 \ge   \int_{C} |j-G|^2  +  2\pi  \frac{d_B}{r_B}  - 2 \pi  M
\end{equation}
for every concentric circle $C \subset A_i$ for some annulus $A_i \in \mathcal{A}(\eta, r)$.

{\it Step 4:}  Energy estimates

Define $\F(x,r) = \int_{\bar{B}(x,r)} |j|^2$, where $\bar{B}(x,r)$ is the closed ball centered at $x$ of radius $r$.  If $B = \bar{B}(x,r)$, we may then unambiguously write $\F(B) = \F(x,r)$.  For finite, disjoint collections of closed balls, $\mathcal{B}$, we can then define $\F(\mathcal{B}) := \sum_{B \in \mathcal{B}} \F(B)$.

Let $S \subset (n \eta,r]$ denote the finite set of merging times produced in the ball-growth procedure of Lemma \ref{ball_growth} that generated the family $\{\mathcal{B}(t)\}_{t\in [n\eta,r]}$.  Lemma 2.3 of \cite{st}, which is a variant of Proposition 4.1 of \cite{ssbook}, then implies that for every $B \in \mathcal{B}_r$ so that $B \subset U$, we have that
\begin{multline}\label{br_1}
 \F(B) - \F(\mathcal{B}_0 \cap B) \ge \int_{n \eta}^{r} \sum_{ \bar{B}(x,t) \in  \mathcal{B}(s) \cap B}  \frac{\partial \F}{\partial r}(x,t) ds \\
+ \sum_{s \in S} \F(\mathcal{B}(s) \cap B) - \F(\mathcal{B}(s)\cap B)^-,
\end{multline}
where we understand that $\mathcal{B}(s) \cap B = \{ B' \in \mathcal{B}(s) \;\vert\; B'  \subset B\}$ for $s \in [n\eta,r]$, and where we have written $\F(\mathcal{B}(s)\cap B)^- = \lim_{t \to s^-} \F(\mathcal{B}(s)\cap B)$.  Note that we may rewrite the left side of \eqref{br_1} as
\begin{equation}
 \F(B) - \F(\mathcal{B}_0 \cap B) = \int_{B\sm\cup_{p\in\Lambda} B(p,\eta)} \abs{j}^2
\end{equation}
for each $B \in \mathcal{B}_r$ so that $B \subset U$.

The estimate \eqref{circle4} implies that if $B' = \bar{B}(x,t) \in  \mathcal{B}(s) \cap B$, then
 (using again that $d_{B'}$ is a positive integer)
\begin{equation}
\frac{\partial \F}{\partial r}(x,t) \ge   \int_{\partial B'} |j-G|^2+  2\pi d_{B'} \(\frac1{t} -M\).
\end{equation}
We may compute
\begin{equation}
 \int_{n\eta}^{r} \sum_{\bar{B}(x,t) \in  \mathcal{B}(s) \cap B}   \int_{\partial \bar{B}(x,t)} \abs{j-G}^2   ds
= \int_{\cup_{A \in \mathcal{A}(\eta, r)} A } \abs{j-G}^2.
\end{equation}
On the other hand, it is easy to see that
\begin{multline}\label{br_2}
 \sum_{s \in S} \F(\mathcal{B}(s) \cap B) - \F(\mathcal{B}(s)\cap B)^-
= \sum_{s \in S} \left[\sum_{B' \in \mathcal{B}(s) \cap B} \int_{B'} \abs{j}^2 - \lim_{t \to s^-} \sum_{B' \in \mathcal{B}(t) \cap B} \int_{B'} \abs{j}^2 \right]\\
= \int_{B \backslash \left[\left(\cup_{p \in \Lambda} \bar{B}(p,\eta)\right)  \cup \left( \cup_{A \in \mathcal{A}(\eta,r)} A \right)\right] } \abs{j}^2 = \int_{B \backslash \left[\left(\cup_{p \in \Lambda} \bar{B}(p,\eta)\right)  \cup \left( \cup_{A \in \mathcal{A}(\eta, r)} A \right)\right] } \abs{j- G}^2
\end{multline}
since  by construction $G =0$ on ${B \backslash \cup_{A \in \mathcal{A}(\eta, r)} A }$.  Combining \eqref{br_1}--\eqref{br_2} and writing $r_{B'}$ for the radius of a ball $B'$, we then deduce that
\begin{equation}\label{br_3}
\int_{B\sm\cup_{p\in\Lambda}\bar{B}(p,\eta)}|j|^2 \ge \int_{B\sm\cup_{p\in\Lambda}\bar{B}(p,\eta)} |j-G|^2
+   \int_{n\eta}^r \sum_{ B'\in \B(s) \cap B}  2\pi d_{B'} \left( \frac{1}{r_{B'}} - M   \right) \,ds.
\end{equation}
Since
\begin{equation*}
\sum_{B' \in \B(s)\cap B} r_{B'} \le  \sum_{B' \in \B(s)} r_{B'} =s \text{ and }  \sum_{B' \in \B(s)\cap B} d_{B'} =  \#(B\cap\Lambda) = n_B,
\end{equation*}
we may estimate
\begin{multline}\label{br_4}
 \int_{n\eta}^{r} \sum_{ B'\in \B(s) \cap B}  2\pi d_{B'} \left( \frac{1}{r_{B'}} - M   \right) \,ds \ge
\int_{n\eta}^{r} \sum_{ B'\in \B(s) \cap B}  2\pi d_{B'} \left( \frac{1}{s} - M   \right) \,ds \\
=
2\pi n_B \int_{n\eta}^{r}   \left( \frac{1}{s} - M   \right) \,ds
=  2\pi n_B \(\log\frac r{n\eta} - M(r- n\eta)\).
\end{multline}
We may then use \eqref{br_4} in \eqref{br_3} to deduce the estimate \eqref{br_01}, which proves the fourth item.

{\it Step 5:}  Proof of the fifth item

Let $B \in \mathcal{B}_r$ and assume that $0 < \eta < \min\{\eta_0,r/n\}$.   Set $B_\eta = B\sm\cup_{p\in\Lambda} \bar{B}(p,\eta)$. Then by the ``layer-cake'' theorem (see Theorem 1.13 of \cite{lieb_loss}), for any continuous non-negative $\chi$,
\begin{equation}\label{machin}
\int_{B_\eta} \chi |j|^2  = \int_0^{+\infty} \(\int_{B_\eta\cap\{\chi>t\}}|j|^2\)\,dt.
\end{equation}
Now, if $p\in\Lambda\cap B$,  then for any $s\in(0,r]$  there exists a closed  ball $B_{p,s}\in\B_s$ containing $p$. For $t>0$ we call
$$ s(p,t) = \sup \{ s \in (0, r], B_{p,s}\subset \{\chi>t\} \}$$
if this set is nonempty, and let $s(p,t) = 0$ otherwise, i.e. if $\chi(p)\le t$. Then for those $p,t$ so that $s(p,t) >0$, we  let $B_p^t = B_{p,s(p,t)}.$  Note that $p$ is not necessarily the center of $B_p^t$, and also that $s(p,t)$ bounds from above the radius of $B_p^t$, but is not necessarily equal to it.

As noted above,   $s(p,t) = 0$ iff $\chi(p) \le t$, while if  $s(p,t)\in(0,r)$ then $B_p^t\not\subset \{\chi>t\}$, otherwise there would exist $s'>s(p,t)$ such that  $B_{p,s'}\subset \{\chi>t\}$, contradicting the definition of $s(p,t)$. Thus, choosing $y$ in $B_p^t\sm \{\chi>t\}$, we have
\begin{equation}\label{boundr}
\chi(p) - t\le \chi(p) - \chi(y)\le 2 s(p,t)\pnorm{\nab\chi}{\infty}.
\end{equation}
Also, for any $t\ge 0$ the collection $\{B_p^t\}_p$, where $p\in\Lambda$ and the $p$'s for
which $s(p,t) = 0$ have been excluded, is disjoint. Indeed if  $p,b\in \Lambda$ and $s(p,t)\ge s(b,t)$ then, since $\B_{s(p,t)}$ is disjoint, the balls $B_{p,s(p,t)}$ and
$B_{b,s(p,t)}$ are either equal or disjoint. If they are disjoint, we
note that $s(p,t)\ge s(b,t)$ implies that $B_{b, s(b,t)}\subset
B_{b,s(p,t)}$, and therefore $B_b^t = B_{b,s(b,t)}$ and $B_p^t =
B_{p,s(p,t)}$ are disjoint. If they are equal, then
$B_{b,s(p,t)}\subset E_t\cap B$, and therefore $s(b,t)\ge s(p,t)$,
which implies $s(b,t) = s(p,t)$ and then $B_b^t = B_p^t$.

Now assume that  $B'\in \{B_p^t\}_p$, write $n = \#\Lambda$, and let $s$ be the common value of $s(p,t)$ for $p$'s in $B'\cap\Lambda$.  Then the fourth item of the proposition yields for any  $\eta<\min(\eta_0,r/n)$  (but the inequality is trivially true if $\eta>r/n$),
\begin{equation}
\int_{B'\sm\cup_{p\in\Lambda}\bar{B}(p,\eta)} |j|^2\ge \nu(B')\(\log\frac s{n\eta} - M s \)_+  + \int_{B'\sm\cup_{p\in\Lambda}\bar{B}(p,\eta)} |j- G|^2.
\end{equation}
We may rewrite the above as
\begin{equation} \\
\int_{B'\sm\cup_{p\in\Lambda}\bar{B}(p,\eta)} |j|^2\ge 2\pi\sum_{p\in B'\cap\Lambda}\(\log\frac {s(p,t)}{n\eta} - M s(p,t) \)_+ + \int_{B'\sm\cup_{p\in\Lambda}\bar{B}(p,\eta)} |j- G|^2.
\end{equation}
Summing over $B'\in \{B_p^t\}_p$ , we deduce (since the $p$'s for which $s(p,t)=0$ do not contribute to the sum)
\begin{multline}\label{p11}
\int_{  \left(\cup_{p\in B\cap \Lambda } B_p^t  \right) \sm \left(\cup_{p \in\Lambda}\bar{B}(p,\eta)\right)   } |j|^2\ge 2\pi\sum_{p\in B\cap\Lambda}\(\log\frac {s(p,t)}{n\eta} - M s(p,t) \)_+    \\ +
\sum_{p\in B\cap\Lambda}  \int_{     B_p^t \sm\cup_{p\in\Lambda} \bar{B}(p,\eta)} |j- G|^2.
\end{multline}
On the other hand, by simple algebra, for any $\beta>0$ there exists $C_\beta>0$ such that
\begin{multline}
\int_{   B_\eta\cap  \{\chi > t\} \sm \( \cup_{p\in B\cap \Lambda } B_p^t\)}  |j|^2 \\
\ge
\frac{1}{1+\beta} \int_{   B_\eta\cap  \{\chi > t\} \sm \( \cup_{p\in B\cap \Lambda}  B_p^t\)}  |j-G|^2 - C_\beta \int_{   B_\eta\cap  \{\chi > t\} \sm \( \cup_{p\in B\cap \Lambda }B_p^t\)}  |G|^2.
\end{multline}
Adding this  to \eqref{p11}, we are led to
\begin{multline}\label{p12}
\int_{   B_\eta\cap  \{\chi > t\}    } |j|^2\ge 2\pi\sum_{p\in B\cap\Lambda}\(\log\frac {s(p,t)}{n\eta} - M s(p,t) \)_+   \\ + \frac{1}{1+\beta}
\int_{   B_\eta\cap  \{\chi > t\}    } |j-G|^2 - C_\beta
\int_{   B_\eta\cap  \{\chi > t\} \sm \( \cup_{p\in B\cap \Lambda } B_p^t \)} |G|^2.
\end{multline}

We now turn to estimating
\begin{equation*}
\int_{ B_\eta \cap  \{\chi > t\} \sm \( \cup_{p\in B\cap \Lambda } B_p^t \) } |G|^2.
\end{equation*}
By the definition of $G$, we may rearrange the annuli on which $G$ is supported to arrive at the estimate
\begin{multline}
\int_{   B_\eta\cap  \{\chi > t\} \sm \( \cup_{p\in B\cap \Lambda } B_p^t \)} |G|^2
\le  \sum_{\substack{p \in B\cap \Lambda\\ s(p,t)>0}} \int_{B[p,r]\backslash B[p, s(p,t)]} \frac{dx}{|x-p|^2} \\
=  \sum_{\substack{p\in B\cap \Lambda\\ s(p,t)>0}} \log \frac{r}{s(p,t)}
= \sum_{\substack{p\in B\cap \Lambda\\ s(p,t)>0 }} \log \left(\frac{r}{s(p,t)} \vee 1\right).
\end{multline}
Inserting \eqref{boundr}, we obtain
\begin{equation}
\int_{   B_\eta\cap  \{\chi > t\} \sm \( \cup_{p\in B\cap \Lambda}  B_p^t \)} |G|^2
\le \sum_{\substack{p\in B\cap \Lambda\\ s(p,t)>0 }}
  \log \( \frac{2r \pnorm{\nab \chi }{\infty}  }{(\chi(p) -t)_+}\vee 1 \).
\end{equation}
We now integrate this over $t$, which yields
\begin{multline}\label{p112}
\int_0^\infty \int_{   B_\eta\cap  \{\chi > t\} \sm \( \cup_{p\in B\cap \Lambda } B_p^t \)} |G|^2 \, dt
\le  \sum_{p \in B\cap \Lambda } \int_{\chi(p) - 2r \pnorm{ \nab \chi }{\infty} }^{\chi(p)}\log
 \frac{2r \pnorm{\nab \chi }{\infty}  }{\chi(p) -t}\, dt  \\
= \sum_{p \in B\cap \Lambda } 2 r \pnorm{\nab \chi}{\infty},
\end{multline}
where for the last equality we have used the change of variables $v=\frac{\chi(p)-t}{2r \pnorm{\nab \chi}{\infty}}$.

Similarly,  using \eqref{boundr} and the fact that $s(p,t)\le r$, we have
\begin{multline}\label{p113}
\int_0^\infty
2\pi\sum_{p\in B\cap\Lambda}\(\log\frac {s(p,t)}{n\eta} - M s(p,t) \)_+\, dt
 \\ \ge
2\pi\sum_{p\in B\cap\Lambda}\int_0^{\chi(p)} \(\log\frac r{n\eta} +\log \(\frac{\chi(p) - t}{2r \pnorm{\nab\chi}{\infty}} \wedge 1\) - M r  \)\,dt \\ \ge
2\pi\sum_{p\in B\cap\Lambda}\(\chi(p)\(\log\frac r{n\eta} -Mr \) - 2r \pnorm{ \nab\chi}{\infty} \).\end{multline}
Integrating \eqref{p12} with respect to $t$ and combining with \eqref{machin},  \eqref{p112}, and \eqref{p113} (modifying $C_\beta $ if necessary), we are led to
\begin{equation*}
\int_{B_\eta}  \chi |j|^2 \ge   2\pi\sum_{p\in B\cap\Lambda}\chi(p)\(\log\frac r{n\eta} -Mr \)  - C_\beta r \pnorm{\nab \chi}{\infty} \nu (B) + \frac{1}{1+\beta} \int_{B_\eta} \chi |j-G|^2,
\end{equation*}
which is \eqref{br_02}.  This proves the fifth item and completes the proof.
\end{proof}

We are now in a position to finish the
\begin{proof}[Proof of Theorem \ref{main}]

We proceed through several steps.

{\it Step 1:} Localizing the ball construction

The first step is to use a covering to localize the ball construction estimates of Proposition \ref{boulesren}.  Our method follows that of Proposition 4.8 of \cite{ss1}, which was based on the method used in \cite{compagnon}.  We cover $\mr^2$ by the  balls of radius $1/4$ whose centers are in $\frac{1}{8} \mz^2$.  We call this cover  $\{U_\alpha\}_\alpha$ with $\{x_\alpha\}_\alpha$ the centers.  For each $\alpha$ so that $U_\alpha\cap\widehat U\neq \varnothing$ and for any $r\in(0,1/4)$  we construct disjoint    balls $\B_r^\alpha$ and vector fields $G_r^\alpha: \cup_{B \in \B_r^\alpha} B \to \Rn{2}$ using Proposition \ref{boulesren}  (here we view the $G_r^\alpha$ as the restrictions to $\cup_{B \in \B_r^\alpha} B$ of the vector fields constructed in  the proposition).

Assume that  $\rho\in(0,1/4)$ (with value to be specified below).  We claim that for each $\alpha$ we can extract a subcollection $\tilde{\B}_\rho^\alpha \subseteq \B_\rho^\alpha$ so that $\mathcal{B}_\rho := \cup_\alpha \tilde{\B}_\rho^\alpha$ is a disjoint  cover of $\Lambda$.   To prove the claim, we first note that if  $\mathcal{C}$ is a connected component of $\cup_\alpha  \B_\rho^\alpha$,  then there exists  $\alpha_0$ such that $\mathcal{C}\subset U_{\alpha_0}$ (for the proof, see \cite{ss1} Proposition 4.8).  Then, to obtain a disjoint cover of $\Lambda$ from $\cup_\alpha \B_\rho^\alpha$, we let $\mathcal{C}$ run over all the  connected components of  $\cup_\alpha  \B_\rho^\alpha$, and for a given $\mathcal{C}$ such that $\mathcal{C}\subset U_{\alpha_0}$,  we remove from $\mathcal{C}$ the balls which do not belong to $\B_\rho^{\alpha_0}$.   We let $\tilde{\B}_\rho^\alpha$ to denote the family with deleted balls, and let $\B_\rho =\cup_\alpha \tilde{\B}^\alpha_\rho$. Then $\B_\rho$ covers $\Lambda$ and is disjoint.  Also, each ball in $\tilde{\B}_\rho^\alpha$ is contained in $U_\alpha$.   This proves the claim.

{\it Step 2:}  Introducing $G$

Let us write $\eta_0 = \hal \min\{\abs{p-q} \;\vert\; p,q \in \Lambda, p \neq q  \}$ and $n_\alpha = \nu(U_\alpha) \le \nu(\widehat{U}) = n$.  We set
\begin{equation*}
\gamma_\rho := \min\{\eta_0,\rho/n \} \le \min\{ \eta_0, \rho/n_\alpha \}.
\end{equation*}
According to Proposition \ref{boulesren}, if $0 < \eta \le \gamma_\rho$, then we can view $\B_\rho^\alpha$ as having been generated by a ball-growth, starting with $\{\bar{B}(p,\eta)\}_{p \in \Lambda \cap U_\alpha}$,  via the family $\{\B^\alpha(t)\}_{t \in [n_\alpha \eta, \rho]}$.  For any  $0<\eta\le \gamma_\rho$ we then write $\mathcal{A}^\alpha(\eta,\rho)$ for the collection of disjoint annuli generated from $\{\B^\alpha(t)\}_{t \in [n_\alpha \eta, \rho]}$ according to Definition \ref{b_growth_annuli}.   Note that the construction of $G^\alpha_\rho$ guarantees that for $x \in \left( \cup_{B \in \B_\rho^\alpha} B \right)$,
\begin{equation*}
 G_\rho^\alpha(x) = \sum_{A \in \mathcal{A}^\alpha(\eta,r)} \vchi_{A}(x) \frac{(x-c_A)^\bot}{\abs{x-c_A}^2} + \sum_{p \in \Lambda \cap U_\alpha} \vchi_{\bar{B}(p,\eta)}(x)\frac{(x-p)^\bot}{\abs{x-p}^2},
\end{equation*}
where $c_A \in \Rn{2}$ denotes the center of the annulus $A$.

Now we let  $\tilde{\mathcal{A}}^\alpha(\eta,\rho)$ denote the  collection of  annuli in $\mathcal{A}^\alpha(\eta,\rho)$ that are contained in one of the balls in $\tilde{\B}^\alpha_\rho$.  Then  $\mathcal{A}(\eta,\rho) := \cup_\alpha \tilde{\mathcal{A}}_\rho^\alpha$  is a finite,  disjoint collection of annuli, each of which is contained in  $\left(\cup_{B \in \B_\rho}B \right) \backslash \cup_{p \in \Lambda} \bar{B}(p,\eta) $.  We now define $G : \left(\cup_{B \in \B_\rho} B \right)\to \Rn{2}$ by
\begin{equation*}
 G(x) = \sum_{\alpha}   \sum_{A \in \tilde{\mathcal{A}}^\alpha(\eta,\rho)} \vchi_{A}(x) \frac{(x-c_A)^\bot}{\abs{x-c_A}^2}   + \sum_{p \in \Lambda} \vchi_{\bar{B}(p,\eta)}(x)\frac{(x-p)^\bot}{\abs{x-p}^2}.
\end{equation*}
We then extend $G$ by $0$ on $\widehat{U} \backslash \left(\cup_{B \in \B_\rho} B \right)$ to view $G: \widehat{U} \to \Rn{2}$.  Clearly, $G(x) = G_\rho^\alpha(x)$ for all $x \in \cup_{B \in \tilde{\B}_\rho^\alpha} B$.

It is clear that $\sum_{\a}n_\a\le C_* n$ where $C_*< \infty$ is the overlap number of the $U_\a$'s, defined as the maximum number of sets to which any $x$ belongs.  We will use this fact to estimate  $\|G\|_{\lti}$.    We combine the first, third, and fourth  items of Lemma \ref{quasi_properties}, the estimate \eqref{normed},  and Proposition  \ref{lorentz_est} to see that
\begin{equation}
\begin{split}
 \hal \pqnorm{G}{2}{\infty}^2 & \le \wnorm{G}^2 \\
& \le \sum_{\alpha} \wnorm{ \sum_{A \in \tilde{\mathcal{A}}_\rho^\alpha} \vchi_{A}(\cdot) \frac{(\cdot-c_A)^\bot}{\abs{\cdot-c_A}^2} }^2
+ \sum_{p \in \Lambda} \wnorm{\vchi_{\bar{B}(p,\eta)}(\cdot) \frac{(\cdot-p)^\bot}{\abs{\cdot-p}^2}}^2  \\
&\le \sum_{\alpha} \wnorm{ \sum_{A \in \mathcal{A}_\rho^\alpha} \vchi_{A}(\cdot) \frac{(\cdot-c_A)^\bot}{\abs{\cdot-c_A}^2} }^2
+ \sum_{p \in \Lambda} \wnorm{ \frac{1}{\abs{\cdot-p}}}^2
\\
&\le \sum_{\alpha} \pqnorm{ \sum_{A \in \mathcal{A}_\rho^\alpha} \vchi_{A}(\cdot) \frac{(\cdot-c_A)^\bot}{\abs{\cdot-c_A}^2} }{2}{\infty}^2
+ \sum_{p \in \Lambda} \pi
\\
&\le \sum_\alpha 4 \pi n_\alpha  +  n \pi \le 4\pi C_* n + n\pi = \pi(4C_* +1) n.
\end{split}
\end{equation}
Hence $\pqnorm{G}{2}{\infty}^2 \le C n$, which  proves the first item of the theorem.

{\it Step 3:} Preliminaries for the main estimate

We now turn to the proof of the main estimate, \eqref{jW}.  The last  item of Proposition \ref{boulesren}, applied to a ball $B \in \tilde{\B}_\rho^\alpha$, guarantees that if $0<\eta\le \gamma_\rho$   and $B_\eta := B\sm\cup_{p\in\Lambda}\bar{B}(p,\eta)$, then for any $\beta>0$ and any non-negative function $\chi$ vanishing outside $\widehat{U}$ we have
\begin{multline}
\frac{1}{1+\beta} \int_{B_\eta} \chi |j -G|^2
 \le \int_{B_\eta}\chi |j|^2
 - 2\pi\(\log\frac \rho{n_\alpha\eta} - \rho M\)\sum_{p\in B\cap\Lambda} \chi(p)\\
+ C_\beta    \rho \nu(B) \pnormspace{\nab\chi}{\infty}{B}.
\end{multline}
We restrict to the $\alpha$'s such that  $U_\a$ intersects $\supp (\chi)$ and then sum over $B \in   \B_\rho$;  since  $G$ vanishes outside $\cup_{B \in \B_\rho}B$, we deduce that for $U(\eta):= U \backslash \cup_{p \in \Lambda} \bar{B}(p,\eta)$,
\begin{multline*}\frac{1}{1+\beta}
\int_{U(\eta)} \chi  |j- G|^2 \le C_\beta \rho   \sum_\a  n_\a \pnorm{\nab \chi }{\infty} \\
+ \sum_{B \in \B_\rho} \left[ \int_{B\sm\cup_{p\in\Lambda}\bar{B}(p,\eta)} \chi |j|^2 - 2\pi\log\frac {1}{\eta}
 \sum_{p\in B\cap\Lambda} \chi(p) \right]
 \\
+  \sum_{\{\a \;\vert\; U_\a \cap \supp(\chi)  \neq \varnothing\}}
  n_\a(    2\pi \rho  M + 2\pi \log (2n_\a))   \pnorm{\chi}{\infty}
\end{multline*}
 Letting $\eta\to 0$, in view of the definition of $W(j, \chi)$ \eqref{WR},  we find
\begin{multline}
\limsup_{\eta \to 0}\frac{1}{1+\beta}   \int_{U(\eta)} \chi  |j- G|^2
\le 2 W(j, \chi) +  C_\beta  C_* \rho   n \pnorm{\nab \chi}{\infty}   \\
+ \sum_{\{\a \;\vert\; U_\a \cap \supp(\chi)  \neq \varnothing\}}  n_\a(2  \pi \rho   M + 2\pi \log (2n_\a))   \pnorm{\chi}{\infty}.
\end{multline}
 It follows that $\chi |j-G|^2 \in L^1 (U) $ and, changing the constants if necessary,
 \begin{multline}\label{circle5}
 \hal \int_U  \chi|j- G|^2 \le (1+\beta)  W(j, \chi) +  C_\beta   n \pnorm{\nab \chi}{\infty} \\   +C\sum_{\{\a \;\vert\; U_\a \cap \supp(\chi)  \neq \varnothing\}}   n_\a (1+  \log n_\a )  \pnorm{\chi}{\infty}.
\end{multline}

{\it Step 4:}  Completing the main estimate

This step again  follows \cite{ss1}.  There exists  a number   $k$ which bounds the number of $\beta$'s such that  $\dist(U_\beta,   U_\alpha)< 1/2$ for any given $\alpha$. Therefore, the total radius of the balls in $\B_\rho$ that are at distance less than $1$ from $U_\alpha$ is at most $k\rho$. We may then choose $\rho$ small enough that $k\rho <\frac{1}{16}$.  Then, letting $T_\alpha$ denote the set of $t\in (0,\frac34)$ such that the circle of center $x_\alpha$ (where we recall $x_\alpha$ is the center of $U_\alpha$)  and radius $t$ does not intersect $\tilde{\B}_\rho^\alpha$, we have $|T_\alpha|\ge 3/4 - 1/16= 11/16$. Moreover, if $U_\alpha\cap U\neq\varnothing$ then $d(x_\alpha ,U ) \le 1/4$, and hence $\bar{B}(x_\alpha,3/4)\subset\widehat U$. In particular, letting $C_\alpha = \{x\mid |x-x_\alpha|\in T_\alpha\}$, we have $C_\alpha\subset \widehat U$.   Then  there exist universal constants $c>0$ and $C$ such that
\begin{equation}\label{lacouronne}
\int_{C_\alpha}|j|^2 \ge c {n_\alpha}^2 - CM^2.
\end{equation}
To see this, we apply \eqref{circle1} on the circle $S_t = \{|x-x_\alpha|=t\}$, i.e. with $r_B = t$ and $d_B = \#(\Lambda\cap B(x_\alpha,t))$. Using the fact that $d_B\ge n_\alpha$ and $t\in(0, \frac34)$, as well as the  Cauchy-Schwarz inequality and the relation $(a-b)^2 \ge \frac{a^2}{2}  -b^2$, we deduce that \begin{equation*}
\int_{S_t}|j|^2\ge \frac{(2\pi d_B-\pi m(B) t)^2}{2\pi t} \ge \frac{4\pi }{3} {n_\alpha}^2 - \frac{3\pi}{8}M^2.
\end{equation*}
Integrating this with respect to $t\in T_\alpha$ yields \eqref{lacouronne}.
Note that $G = 0$  in each  $C_\alpha$ by construction, so we may deduce from \eqref{lacouronne}  that $ \int_{C_\alpha}|j-G|^2 \ge c {n_\alpha}^2 - CM^2$.
Finally, modifying $C$, we may change this relation into
\begin{equation}\label{lacouronne2}
\int_{C_\alpha}|j-G|^2 \ge c \( {n_\alpha}^2 - Cn_\alpha\)_+.
\end{equation}
Indeed, if $n_\alpha =0$ the relation is trivially true, and if not then we have $n_\alpha \ge 1$.

Let us write $k'$ for  the overlap number of the sets $\{C_\alpha\}_\alpha$, defined as the maximum number of sets to which any $x$ belongs. It is bounded by the overlap number of $ \{B(x_\alpha,3/4)\}_\alpha$. Since
\begin{equation}
\int_U \chi |j-G|^2 \ge (1-\beta) \int_U \chi |j-G|^2 + \beta \frac{1}{k'} \sum_{\{\a \;\vert\; U_\a \cap \supp(\chi)  \neq \varnothing\}}   \int_{C_\a} \chi |j-G|^2,
\end{equation}
we deduce from \eqref{circle5} and \eqref{lacouronne2} that
 \begin{multline*}
 \frac{\beta c}{2 k'} \sum_{\{\a \;\vert\; \chi \ge \hal \pnorm{\chi}{\infty}   \text{ on } \ U_\a \}}  (\min_{C_\a} \chi)( n_\a^2- C n_\a  )_+ + \frac{1-\beta}{2}\int_U \chi |j-G|^2
 \\ \le
 (1+\beta) W(j, \chi) +  C_\beta    n \pnorm{\nab \chi }{\infty}   +C \sum_{\{\a \;\vert\; U_\a \cap \supp(\chi)  \neq \varnothing\}}   n_\a (1+  \log n_\a )  \pnorm{\chi}{\infty}.
\end{multline*}
If $U_\a  \subset \{\chi \ge \hal \pnorm{\chi}{\infty}\}$  we have that
\begin{equation*}
\frac{\beta c}{ 2 k'}  (\min_{C_\a} \chi)( n_\a^2- C n_\a  )_+ - C n_\a(1+ \log n_\a)\pnorm{\chi}{\infty}  \ge - C_\beta n_\a \pnorm{\chi}{\infty},
\end{equation*}
where $C_\beta$ depends only on $\beta$ and $M$. For the other $\a$'s, $U_\a$ intersects $\supp(\chi) \cap \{\chi\le \hal \pnorm{\chi}{\infty} \}$, and since the diameter of the $U_\a$'s are bounded by  $\hal$ (by construction), we may bound
\begin{equation*}
\sum_{\{\a \;\vert\; U_\a \cap \supp(\chi) \cap \{\chi \le \hal \pnorm{\chi}{\infty} \}  \neq \varnothing\}}   n_\a (1+  \log n_\a ) \le 2 n'\log n',
\end{equation*} 
where $n'= \#\{p\in\Lambda\mid B(p,\hal)\cap \{0< \chi \le \hal \pnorm{\chi}{\infty} \} \neq\varnothing\}$.
We are led to
\begin{equation}
 \frac{1-\beta}{2} \int_U \chi |j-G|^2 \le
 (1+\beta) W(j, \chi) +  C_\beta    n \pnorm{\nab \chi}{\infty} + C_\beta n \pnorm{\chi}{\infty}  +    C n'\log n',
\end{equation}
which yields the estimate \eqref{jW}, after changing $\beta$ into $\beta/2$.
\end{proof}

With Theorem \ref{main} in hand, we now conclude with  the
\begin{proof}[Proof of the Corollary \ref{main_cor}]
Using  the embedding relation $\|f\|_{\lti(U) }\le \|f\|_{L^2(U)}$, we deduce from the second item of Theorem \ref{main}, applied with $\beta=1$, that
\begin{equation}\label{cor1}
\|\sqrt{\chi } (j-G) \|_{\lti (U)} \le C \( W(j, \chi) + n (\pnorm{\chi}{\infty}+\pnorm{\nab \chi}{\infty} )  + C n'\log n'\)^\hal.
\end{equation}
We may then estimate $\sqrt{\chi} G$ by combining the first item of Lemma \ref{quasi_properties}, estimate \eqref{normed}, and the first item of Theorem \ref{main} to see that $\|\sqrt{\chi} G\|_{\lti(U)} \le C (\pnorm{\chi}{\infty} n)^\hal.$   Then from this,  \eqref{cor1}, and the triangle inequality for the $\lti $ norm, we are led to (changing the constants if necessary)
\begin{equation}
\|\sqrt{\chi} j\|_{\lti(U)} \le C \( W(j, \chi) + n (\pnorm{\chi}{\infty}+\pnorm{\nab \chi}{\infty}  )  + C n'\log n'\)^\hal.
\end{equation}
Finally, to conclude the proof we  use the embedding  (see e.g. \cite{grafakos})
\begin{equation*}
\|f\|_{L^p(U)} \le C_p|U|^{\frac{1}{p}-\hal} \|f\|_{\lti(U)}
\end{equation*}
for $1\le p < 2$ and $C_p = (2/(2-p))^{1/p}$, applied to $f=\sqrt{\chi} j$.

\end{proof}

\noindent
{\sc Sylvia Serfaty}\\
UPMC Univ. Paris 06, UMR 7598 Laboratoire Jacques-Louis Lions,\\
 Paris, F-75005 France;\\
 CNRS, UMR 7598 LJLL, Paris, F-75005 France \\
 \&  Courant Institute, New York University\\
251 Mercer St., New York, NY  10012, USA\\
{\tt serfaty@ann.jussieu.fr}
\vskip 1cm
\noindent
{\sc Ian Tice}\\
Brown University, Division of Applied Mathematics\\
 182 George St., Providence, RI 02912, USA \\
{\tt tice@dam.brown.edu}


\begin{thebibliography}{99}
\bibitem{bbh} F. Bethuel, H. Brezis, F. H\'elein. \emph{Ginzburg-Landau Vortices}. Birkh\"auser, Boston, MA, 1994.

\bibitem{forrester} P. J. Forrester. \emph{Log-Gases and Random Matrices}. Princeton University Press, Princeton, NJ, 2010.


\bibitem{grafakos} L. Grafakos. \emph{Classical and Modern Fourier Analysis}. Pearson Education, Inc., Upper Saddle River, NJ, 2004.

\bibitem{je}R. L. Jerrard.  Lower bounds for generalized Ginzburg-Landau functionals.  \emph{SIAM J.  Math. Anal.}  {\bf 30}  (1999), no. 4, 721-746.

\bibitem{lieb_loss} E. Lieb, M. Loss. \emph{Analysis.} Second edition. Graduate Studies in Mathematics, 14. American Mathematical Society, Providence, RI, 2001.

\bibitem{mattila} P. Mattila. \emph{Geometry of Sets and Measures in Euclidean Spaces: Fractals and Rectifiability}. Cambridge Studies in Advanced Mathematics, 44. Cambridge University Press, Cambridge, 1995.

\bibitem{sa} E. Sandier. Lower bounds for the energy of unit vector fields and applications.  \emph{J. Funct. Anal.} {\bf  152}  (1998),  no. 2, 379--403.

\bibitem{ssbook} E. Sandier, S. Serfaty.  \emph{Vortices in the Magnetic Ginzburg-Landau Model}. Birkh\"auser, Boston, MA, 2007.

\bibitem{ss1} E. Sandier, S. Serfaty.  From the Ginzburg-Landau model to vortex lattice
problems.   Preprint, 2010.

\bibitem{compagnon} E. Sandier, S. Serfaty. Improved lower bounds for Ginzburg-Landau energies via mass displacement. To appear in {\it Analysis and PDE}.

\bibitem{ss2} E. Sandier, S. Serfaty. Two-dimensional  log gases  and the renormalized energy.  In preparation.

\bibitem{ss3} E. Sandier, S. Serfaty.  One-dimensional log gases and the renormalized energy.  In preparation.

\bibitem{st} S. Serfaty, I. Tice.  Lorentz space estimates for the Ginzburg-Landau energy.  \emph{J. Funct. Anal.} {\bf 254} (2008), no. 3, 773--825.
\end{thebibliography}
\end{document}